\documentclass[11pt]{article}
\usepackage{amsmath,amsthm,amsfonts,amssymb,mathrsfs,epsfig,bm}
\usepackage[usenames]{color}
\usepackage[colorlinks=true]{hyperref}

\parskip        \smallskipamount
\newtheorem{theorem}{Theorem}
\newtheorem{lemma}{Lemma}
\newtheorem{corollary}{Corollary}
\newtheorem{proposition}{Proposition}

\theoremstyle{remark}

\newtheorem{remark}{Remark}

\def\N{\mathbb{N}}

\def\R{\mathbb{R}}

\def\P{\mathbb{P}}
\def\E{\mathbb{E}}

\renewcommand{\phi}{\varphi}
\renewcommand{\epsilon}{\varepsilon}
\newcommand{\1}{{\text{\Large $\mathfrak 1$}}}

\renewcommand{\limsup}{\varlimsup}
\renewcommand{\liminf}{\varliminf}

\definecolor{mygray}{gray}{0.9}
\definecolor{deeppink}{RGB}{255,20,147}
\definecolor{mygreen}{rgb}{0.05, 0.576, 0.03}
\definecolor{myred}{rgb}{0.768, 0.09, 0.09}

\long\def\symbolfootnote[#1]#2{\begingroup
\def\thefootnote{\fnsymbol{footnote}}\footnote[#1]{#2}\endgroup}
\newcommand{\keywords}[1]{ \noindent {\footnotesize
             {\small \em Keywords and phrases.} {\sc #1} } }
\newcommand{\ams}[2]{  \noindent {\footnotesize
             {\small \em AMS {\rm 2010} subject classification.
             {\rm Primary {\sc #1}; secondary {\sc #2}} } } }

\usepackage{titlesec,fancyhdr,titletoc}
\def\A{\mathsf A}
\def\B{\mathsf B}
\renewcommand{\S}{\mathsf S}

\allowdisplaybreaks

\begin{document}


\title{\Large \bf Laplace transform asymptotics
and large deviation principles for longest success runs in Bernoulli trials}

\author{{\sc Takis Konstantopoulos\thanks{{\tt takiskonst@gmail.com};
Department of Mathematics, Uppsala University, SE-751 06 Uppsala, Sweden;
the work of this author was supported by Swedish Research Council
grant 2013-4688}}
\and
{\sc Zhenxia Liu\thanks{{\tt zhenxia.liu@hotmail.com};
Department of Mathematics,
Link\"{o}ping University, SE-581 83 Link\"{o}ping, Sweden}}
\and
{\sc Xiangfeng Yang\thanks{{\tt xiangfeng.yang@liu.se};
Department of Mathematics,
Link\"{o}ping University, SE-581 83 Link\"{o}ping, Sweden}}}

\date{10 July 2015}
\maketitle

\begin{abstract}
The longest stretch $L(n)$ of consecutive heads in $n$ i.i.d.\ coin tosses
is seen from the prism of large deviations.
We first establish precise asymptotics for the moment generating function
of $L(n)$ and then show that there are precisely two large deviation
principles, one concerning the behavior of the distribution of $L(n)$
near its nominal value $\log_{1/p} n$ and one away from it.
We discuss applications to inference and to logarithmic
asymptotics of functionals of $L(n)$.

\vspace*{2mm}
\ams{60F10,	
44A10,		
60G50		
}
{60G70,		
62F25   	
}
\\[-2mm]
\keywords{Large deviation principle, rate function, Fenchel-Legendre
transform, Laplace transform, moment generating function, runs,  longest run,
Bernoulli trial, confidence interval}
\end{abstract}

\section{Introduction}
The earliest reference to the longest stretch of consecutive successes
in ``random'' trials is (as we learn in the 1981 English translation
\cite[p.\ 138]{VM} of the 1928 book of von Mises)
in a 1916 paper of the German philosopher Karl Marbe
and concerns the longest stretch of consecutive births of children of
the same sex as appearing in the birth register of a Bavarian town.
(This was actually
used by parents to ``predict'' the sex of their child.)
The longest stretch of same-sex births in 200 thousand
birth registrations was actually $17 \approx \log_2 (200\times 10^3)$.
Von Mises \cite{VM1921} was apparently the first one to study
the problem rigorously and his result can be seen in Feller's
Volume I \cite[Section XIII12]{FEL68}.

If $X_1,X_2, \ldots$ are i.i.d.\ Bernoulli trials,
$\P(X_i = 1) = p$, $\P(X_i=0)= q:=1-p$,
and if $L(n)$ is the largest $\ell$ such that
$X_{i+1} + \cdots + X_{i+\ell}=\ell$ for some $0 \le i \le n-\ell$,
then we call the base-$1/p$ logarithm $\log_{1/p} n$ of $n$ the
{\em nominal value} of $L(n)$ because, as Erd\H{o}s and R\'enyi \cite{ER70}
show (in a more general setup in fact; see also
\cite{Erdos-Revesz-1975} and \cite{Renyi-1970}),
\begin{equation}
\label{Lnom}
\lim_{n \to \infty} \frac{L(n)}{\log_{1/p} n} = 1, \quad \text{a.s.}
\end{equation}
The distribution of $L(n)$ is not explicit. Yet, there are many estimates.
The literature is littered with them and one of us recently contributed to it
in \cite{Holst-Konstantopoulos-2015} (where
other quantities, such as the number of times that longest or shortest runs
occur, are also explored).

Our principal interest in this paper is to see to what extent large deviations
theory can be applied to the problem of squeezing something useful
about the distribution of $L(n)$.
We first establish logarithmic asymptotics for the
moment generating function $\E e^{\lambda L(n)}$, as $n \to \infty$.
The asymptotics split in three parts: the subcritical regime,
$\lambda < \ln (1/p)$, the supercritical regime,
$\lambda > \ln (1/p)$, and the critical one when $\lambda=\ln (1/p)$.
These asymptotics can be used in combination with the
G\"artner-Ellis theorem (but see Remark \ref{rem2} below)
to derive a full large deviations principle (LDP).
There are precisely two LDPs. One concerning
the behavior of the distribution of $L(n)$ near its nominal value
$\log_{1/p} n$ and another far away from it.

We outline the results below. Our starting point is
asymptotics for the moment generating function and this is what we do
right away.
Note that we use $\ln$ for natural logarithm and $\log_b$ for logarithm
with base $b$. The symbol $a_n \sim b_n$ means $a_n/b_n \to 1$ as $n
\to \infty$.
Note also that we use the term ``Laplace transform'' interchangeably
with the term ``moment generating function''.
(The variable $\lambda$ ranges over the whole real line.)
\begin{theorem}\label{proposition:Laplace-transform}
The moment generating function of $L(n)$ has the following asymptotics:\\
\noindent (i) Subcritical regime: for $\lambda< \ln(1/p),$
\begin{equation*}
\ln\,\mathbb{E}\exp\left\{\lambda\, L(n)\right\}\sim \lambda
\,\log_{1/p}n;\\
\end{equation*}
\noindent (ii) Critical regime: for $\lambda= \ln(1/p),$
\begin{equation*}
\ln\,\mathbb{E}\exp\left\{\lambda\, L(n)\right\}\sim 2\lambda
\,\log_{1/p}n;\\
\end{equation*}
\noindent (iii) Supercritical regime: for $\lambda> \ln(1/p),$
\begin{equation*}
\ln\,\mathbb{E}\exp\left\{\lambda\, L(n)\right\}\sim
(\lambda-\ln(1/p))\, n.
\end{equation*}
\end{theorem}
To the best of the authors' knowledge, the asymptotics on the
moment generating function in Theorem \ref{proposition:Laplace-transform} have
not explicitly appeared in the literature.
To show Theorem \ref{proposition:Laplace-transform}
there are several options. One option is the use of
the recursion formula
\[
\mathbb{E}\exp\left\{\lambda\,
L(n)\right\}=q\sum_{j=0}^{n-1}p^j\mathbb{E}\exp\left(\lambda\,
\max\{L(n-j-1),j\}\right)+p^ne^{\lambda n},
\]
appearing in \cite{Holst-Konstantopoulos-2015}.
Another possible option is to use Fibonacci-type polynomials,
as appearing in the combinatorially-derived expressions for the
moment generating function in \cite{Philippou-1985}.
But the simplest method is a good estimate for the distribution of
$L(n)$; see Lemma \ref{Lbounds}.
Why this lemma works to establish the asymptotics in the subcritical
and critical regimes is the subject of Section \ref{LTA}
(Lemmas \ref{lemma:less} and \ref{lemma:june-1}).

One implication of Theorem \ref{proposition:Laplace-transform}
is that it immediately suggests the form of large deviations of $L(n)$. In
\cite{Fu-Wang-Lou-2003}, a large deviations type probability was established in
the following form
\begin{align}
\label{LDP-from-Fu-Wang-Lou-2003}
\lim_{n\rightarrow\infty}\frac{1}{n}\,\ln \P(L(n)<k)=-\beta,
\end{align}
for a fixed $k$ where $\beta$ is positive constant. Since, however,
$\log_{1/p} n$ is the nominal value of $L(n)$, in the sense that
\eqref{Lnom} holds, the limit \eqref{LDP-from-Fu-Wang-Lou-2003}
is not strictly speaking a result
in the theory of large deviations since it is not about the deviation from the
most probable point $\log_{1/p} n$ of the random variables $L(n)$.
A partial answer was recently included in \cite{Mao-Wang-Wu-2015} who
proved that
\begin{align}
\label{temp-05-28}
\lim_{n\rightarrow\infty}\frac{1}{\log_{1/p}n}\,
\ln\,\P\left(\frac{L(n)}{\log_{1/p}n}\geq1+x\right)=-x\ln(1/p),
\quad x>0.
\end{align}
Despite that the research on head runs is a classical topic with many
applications (see for instance \cite{Balakrishnan-Koutras-2002}),
no explicit general large deviations principles can be
found in the literature.

The subcritical asymptotics of Theorem \ref{proposition:Laplace-transform}
corresponds to the
convergence $L(n)/\log_{1/p}n\rightarrow1$ almost surely as
$n\rightarrow\infty.$
Therefore we can study the large deviations on
$L(n)/\log_{1/p}n.$ Let us first define the function $\Lambda^*(x)$ as
\begin{equation}
\label{rate-function}
\Lambda^*(x)=
\begin{cases}
	+\infty,& x<1,\\
	(x-1)\ln(1/p),& x \ge 1.
\end{cases}
\end{equation}
Notice that $\Lambda^*$ is lower semicontinuous with
$\{x\in \R:\, \Lambda^*(x) \le c\}$ compact for all $c \ge 0$.
This means that $\Lambda^*$ is a {\em good} rate function (in the terminology
of \cite{Dembo-Zeitouni-2009}).
Our references to large deviations theory are
Dembo and Zeitouni \cite{Dembo-Zeitouni-2009}
and Wentzell \cite{Wentzell-1990}.
The following full LDP is obtained as a corollary to
Theorem \ref{proposition:Laplace-transform}.
\begin{corollary}[LDP near the nominal value]
\label{th:longest-head-run}
The normalized longest head run $L(n)/\log_{1/p}n$ satisfies a large deviation
principle with a good rate function $\Lambda^*(x)$ given by
\eqref{rate-function} and speed $\log_{1/p}n$.
Namely, \\
\noindent (i) for any open set $O\subset \mathbb{R},$
\begin{align}
\label{LDP-lower-bound}
\liminf_{n\rightarrow
\infty}\,\,\frac{1}{\log_{1/p}n}\,\,\ln\,\,\mathbb{P}\left(\frac{L(n)}{\log_{1/p}n}\in
O\right)
 \geq-\inf_{x\in O}\Lambda^*(x);
\end{align}
\noindent (ii) for any closed set $F\subset \mathbb{R},$
\begin{align}
\label{LDP-upper-bound}
\limsup_{n\rightarrow
\infty}\,\,\frac{1}{\log_{1/p}n}\,\,\ln\,\,\mathbb{P}\left(\frac{L(n)}{\log_{1/p}n}\in
F\right)
 \leq-\inf_{x\in F}\Lambda^*(x).
\end{align}
\end{corollary}

\begin{remark}
\label{rem1}
Evidently, the large deviation principle presented in Corollary
\ref{th:longest-head-run} generalizes the result \eqref{temp-05-28} in
\cite{Mao-Wang-Wu-2015}, which comes from choosing the open set $O=(1+x,\infty)$
and the closed set $F=[1+x,\infty).$
\end{remark}

\begin{remark}
\label{rem2}
[Connections with the G\"{a}rtner-Ellis theorem] The proof of the large
deviation upper bound \eqref{LDP-upper-bound} comes directly from the
G\"{a}rtner-Ellis theorem (cf.\ \cite{Dembo-Zeitouni-2009}). We note that the
rate function $\Lambda^*$ is the Fenchel-Legendre transform of the following
function
\begin{align*}
\Lambda(\lambda)=
\begin{cases}
+\infty,& \lambda>\ln(1/p),\\
2\lambda,& \lambda=\ln(1/p),\\
\lambda,& \lambda<\ln(1/p),\\
\end{cases}
\end{align*}
that is, $\Lambda^*(x)=\sup_{\lambda\in\mathbb{R}}\left[\lambda
x-\Lambda(\lambda)\right]$.
There is a slight catch here:
to establish the lower bound, the G\"{a}rtner-Ellis theorem requires that
the function $\Lambda$ be essentially smooth, namely,
that $\lim_{k\rightarrow\infty}|\Lambda'(\lambda_k)|=\infty$
as $\lambda_k\rightarrow \ln(1/p)$.
But this is not true here.
Therefore the G\"{a}rtner-Ellis theorem does
not cover our case. If instead we look at the lower bound proposed in the the
G\"{a}rtner-Ellis theorem, then we have for any open set $O,$
\begin{align*}
\liminf_{n\rightarrow
\infty}\,\,\frac{1}{\log_{1/p}n}\,\,\ln\,\,\mathbb{P}\left(\frac{L(n)}{\log_{1/p}n}\in
O\right)
 \geq-\inf_{x\in O\cap H}\Lambda^*(x),
\end{align*}
where $H$ is the so called \textit{set of {\em exposed points}}
\cite[Page 44]{Dembo-Zeitouni-2009} of $\Lambda^*.$ In our
case, it is easy to see that the set $H$ consists of only one point $H=\{1\}.$
So the proposed lower bound from the G\"{a}rtner-Ellis theorem becomes trivial
since
$$\inf_{x\in O\cap H}\Lambda^*(x)=\Lambda^*(1)=0.$$
In summary, our large deviation principle in Theorem \ref{th:longest-head-run}
gives a non-trivial example which the G\"{a}rtner-Ellis theorem does not cover.
\end{remark}

The supercritical regime of Theorem \ref{proposition:Laplace-transform}
gives another large deviation result with a good rate
function $\widetilde{\Lambda}^*(x)$ defined by
\begin{equation}
\label{rate-function-second}
\widetilde{\Lambda}^*(x)=
\begin{cases}
	+\infty,& x<0,	\\
	x\ln(1/p),& 0\le x\le 1,\\
	+\infty,& x>1.
\end{cases}
\end{equation}

\begin{corollary}[LDP away from the nominal value]
\label{th-second:longest-head-run}
The normalized longest head run $L(n)/n$ satisfies a large deviations principle
with a good rate function $\widetilde{\Lambda}^*(x)$ given by
\eqref{rate-function-second} and  speed $n$. Namely,
\\
\noindent (i) for any open set $O\subset \mathbb{R},$
\begin{align*} 
\liminf_{n\rightarrow
\infty}\,\frac{1}{n}\,\ln\,\mathbb{P}\left(\frac{L(n)}{n}\in O\right)
 \geq-\inf_{x\in O}\widetilde{\Lambda}^*(x);
\end{align*}
\noindent (ii) for any closed set $F\subset \mathbb{R},$
\begin{align*}
\limsup_{n\rightarrow
\infty}\,\frac{1}{n}\,\ln\,\mathbb{P}\left(\frac{L(n)}{n}\in F\right)
 \leq-\inf_{x\in F}\widetilde{\Lambda}^*(x).
\end{align*}
\end{corollary}

Another implication of Theorem \ref{proposition:Laplace-transform}
and its corollaries \ref{th:longest-head-run} and
\ref{th-second:longest-head-run}
is in obtaining asymptotics for other functionals of $L(n)$.
We summarize the results as follows.
\begin{corollary}
\label{th-third:general-Laplace}
(I)
If $f: \R_+ \to \R$ is continuous and satisfies one of the two conditions
\begin{align}
\lim_{m\rightarrow\infty}\limsup_{n\rightarrow
\infty}\,\,\frac{1}{\log_{1/p}n}\,\,\ln\,\,\mathbb{E}\left[\exp\left\{\log_{1/p}n\cdot
f(\frac{L(n)}{\log_{1/p}n})\right\}\cdot
\1_{\left\{f(\frac{L(n)}{\log_{1/p}n})\geq
m\right\}}\right]=-\infty,\label{A1}\tag{A.1}\\
\limsup_{n\rightarrow
\infty}\,\,\frac{1}{\log_{1/p}n}\,\,\ln\,\,\mathbb{E}\exp\left\{\log_{1/p}n\cdot\gamma\cdot
f(\frac{L(n)}{\log_{1/p}n})\right\}<\infty,\,\,\,\text{ for some
}\gamma>1,\label{A2}\tag{A.2}
\end{align}
then it holds that
\begin{align*}
\lim_{n\rightarrow
\infty}\,\,\frac{1}{\log_{1/p}n}\,\,\ln\,\,\mathbb{E}\exp\left\{\log_{1/p}n\cdot
f(\frac{L(n)}{\log_{1/p}n})\right\}
 =\max_{x\in\mathbb{R}}[f(x)-\Lambda^*(x)].
\end{align*}
(II)
If $g: \R_+ \to \R$ is continuous and satisfies one of the two conditions
\begin{align}
\lim_{m\rightarrow\infty}\limsup_{n\rightarrow
\infty}\,\,\frac{1}{n}\,\,\ln\,\,\mathbb{E}\left[\exp\left\{n\cdot
g(\frac{L(n)}{n})\right\}\cdot \1_{\left\{g(\frac{L(n)}{n})\geq
m\right\}}\right]=-\infty,\label{B1}\tag{B.1}\\
\limsup_{n\rightarrow
\infty}\,\,\frac{1}{n}\,\,\ln\,\,\mathbb{E}\exp\left\{n\cdot\gamma\cdot
g(\frac{L(n)}{n})\right\}<\infty,\,\,\,\text{ for some
}\gamma>1,\label{B2}\tag{B.2}
\end{align}
then it holds that
\begin{align*}
\lim_{n\rightarrow \infty}\,\,\frac{1}{n}\,\,\ln\,\,\mathbb{E}\exp\left\{n\cdot
g(\frac{L(n)}{n})\right\}
 =\max_{x\in\mathbb{R}}[g(x)-\widetilde{\Lambda}^*(x)].
\end{align*}
\end{corollary}
Here we
list several functions $f$ and $g$ for which
the conclusions of Corollary \ref{th-third:general-Laplace} hold.
The verification is included in Section
\ref{subsec:proof-cor-third}.
\begin{itemize}
\item $f(x)$ and $g(x)$ are continuous and bounded. In this case, (\ref{A1}),
(\ref{A2}), (\ref{B1}) and (\ref{B2}) hold.
\item $f(x)=c x^\alpha$, $x\in \mathbb{R}_+$, where $c>0$ and $0<\alpha<1$.
It is
proved in Section \ref{subsec:proof-cor-third} that (\ref{A1}) holds.
\item $g(x)$ satisfies the condition: there is $m>0$
such that if $|g(x)|\ge m$, then $x>1$.
For instance, with $c_1,c_2,c_3,c_4, \alpha$ positive constants,
the functions
\[
c_1 x^\alpha,\quad
c_2 e^{c_3 x^\alpha},\quad
c_4 \ln(x+\alpha)
\]
satisfy this condition.
Condition (\ref{B1}) is fulfilled for this type of functions  since
$\1_{\left\{g(\frac{L(n)}{n})\ge
m\right\}}\le \1_{\left\{\frac{L(n)}{n}>1\right\}}=0$.
\end{itemize}

Some easy conclusions of Theorem \ref{proposition:Laplace-transform}
concern well-known asymptotics for the moments of $L(n)$.
Formally taking a derivative at $\lambda=0$ of the expression in the subcritical
regime gives
\[
\E L(n)^k \sim (\log_{1/p} n)^k, \quad k \in \N.
\]
The asymptotic expressions of the first two moments can be found in
\cite{Schilling-1990}, and the higher order moments are discussed in
\cite[page 63]{Sandmann-Schonbucher-Sondermann-2002}.
For convenience, we include the
asymptotic mean as follows
\begin{align}
\label{june-10-mean}
\mathbb{E}L(n)=\log_{1/p}n+\log_{1/p}(1-p)+\log_{1/p}(e^\gamma)-\frac{1}{2}+\varepsilon(n)
\end{align}
where $\gamma=0.5772\ldots$ is Euler's constant, and $\varepsilon(n)$ is
``small''.

The rest of the paper is organized as follows. In Section \ref{LTA}
we prove Theorem {\ref{proposition:Laplace-transform},
along with some auxiliary results.
In Section \ref{LDPsec} we prove the large deviation principles,
stated in Corollaries \ref{th:longest-head-run}
and \ref{th-second:longest-head-run}.
Some other asymptotics related to Corollary \ref{th-third:general-Laplace}
are given in \ref{subsec:proof-cor-third}.
We discuss an application to inference in Section \ref{APPLsec}, and some
open problems in Section \ref{OPENsec}.
To save some space we use the abbreviation
\begin{gather*}
\ell(n) := \log_{1/p} n
\end{gather*}
whenever convenient.
As usual, we let
$\lfloor x\rfloor$ to be the largest integer $n$ such that $n \le x$
and $\lceil x\rceil$ to be the smallest integer $n$ such that $n \ge x$.

\section{Laplace transform asymptotics}
\label{LTA}
We obtain logarithmic asymptotics for $\E \exp\{\lambda\, L(n)\}$,
for all $\lambda \in \R$, in several steps.
First, we obtain a lower bound valid for all $\lambda \in \R$.
Then we obtain an upper bound for the subcritical case ($\lambda < \ln(1/p)$).
These two bounds combined give the exact logarithmic asymptotics
for the subcritical case.
The limit in the critical case ($\lambda=\ln(1/p)$) requires special care
and is treated separately.
Finally, we obtain asymptotics for the supercritical case
($\lambda > \ln(1/p)$).

\begin{lemma}
\label{lemma:bigger}
It holds that
\[
\liminf_{n\rightarrow\infty}
\frac{1}{\log_{1/p}n}\ln\, \mathbb{E}\exp\left\{\lambda\, L(n)\right\}
\geq \lambda,
\]
for all $\lambda\in\mathbb{R}$.
\end{lemma}
\begin{proof}
The case $\lambda=0$ is trivial. Assume $\lambda > 0$. Then,
for $0<\epsilon<1$,
\begin{align*}
\E \exp\{\lambda\, L(n)\}
& \ge \E\big[\exp\{\lambda\, L(n)\};\, L(n) \ge (1-\epsilon)\log_{1/p} n\big]
\\
& \ge \exp\{\lambda (1-\epsilon)
\log_{1/p} n\}\, \P(L(n) \ge (1-\epsilon)\log_{1/p} n).
\end{align*}
Hence
\[
\frac{1}{\log_{1/p} n} \, \ln\, \E \exp\{\lambda\, L(n)\}
\ge \lambda(1-\epsilon) + \frac{1}{\log_{1/p} n}
\ln \, \P(L(n) \ge (1-\epsilon)\log_{1/p} n)
\]
Since $\P(L(n) \ge (1-\epsilon)\log_{1/p} n) \to 1$,
\[
\liminf_{n \to \infty}  \frac{1}{\log_{1/p} n} \, \ln\,
\E \exp\{\lambda\, L(n)\}
\ge \lambda (1-\epsilon),
\]
and letting $\epsilon \downarrow 0$ we obtain the result.
When $\lambda < 0$, we use
\[
\E \exp\{\lambda\, L(n)\}
\ge \E\big[\exp\{\lambda\, L(n)\};\, L(n) \ge (1+\epsilon)\log_{1/p} n\big]
\]
and proceed similarly.
\end{proof}

The following bound for the distribution of $L(n)$ is known in
the literature, but we give a simple proof below for completeness.
\begin{lemma}
For all $k,n \in \N$, $1\le k \le n$,
\label{Lbounds}
\[
(1-p^k)^{n-k+1} \le \P(L(n) < k) \le (1-qp^k)^{n-k+1}.
\]
\end{lemma}
\begin{proof}
Let $X_1, X_2, \ldots$ be i.i.d.\ with $\P(X_1=1)=p$, $\P(X_1=0)=q$.
Let $S_i = X_1+\cdots+X_i$, $i \ge 1$.
Notice that $L(n)<k$ if and only if $S_m-S_{m-k}<k$ for
all $k \le m \le n$.
By a standard correlation inequality, 
\[
\P\left(\bigcap_{m=k}^n \{S_m-S_{m-k}<k\}\right) \ge \prod_{m=k}^n
\P(S_m-S_{m-k}<k) = \prod_{m=k}^n (1-p^k) = (1-p^k)^{n-k+1},
\]
and this is the lower bound.
For the upper bound, since, trivially, $L(k-1)<k$, we have
\[
\P(L(n) < k) = \prod_{m=k}^n \frac{\P(L(m)<k)}{\P(L(m-1)<k)}.
\]
But, since, trivially again, $L(m) \ge L(m-1)$ for all $m$,
\[
\P(L(m-1)<k) = \P(L(m)<k) + \P(L(m-1)<k \le L(m)),
\]
and observe that
\begin{multline*}
\P(L(m-1)<k \le L(m)) = \P(L(m-k-1)<k,\, X_{m-k}=0, X_{m-k+1}=\cdots=X_m=1)
\\
= \P(L(m-k-1)<k)\, qp^k  \ge \P(L(m-1)<k)\, qp^k.
\end{multline*}
Substituting this into the previous display gives
$\P(L(m)<k)\le (1-qp^k)\, \P(L(m-1)<k)$
which implies that $\P(L(n) < k) \le \prod_{m=k}^n  (1-qp^k)
= (1-qp^k)^{n-k+1}$, as claimed.
\end{proof}

We next obtain an upper bound in the subcritical regime.
Remember that $\ell(n)=:\log_{1/p} n$.

\begin{lemma}
\label{lemma:less}
It holds that
\[
\limsup_{n\rightarrow\infty}\, \frac{1}{\log_{1/p}n}\ln\,
\mathbb{E}\exp\left\{\lambda\,
L(n)\right\}\leq \lambda,
\]
for $-\infty < \lambda<\ln(1/p)$.
\end{lemma}
\begin{proof}
Suppose first that $0 < \lambda < \ln(1/p)$, pick $\epsilon > 0$, and write
\begin{equation}
\label{dec}
\E e^{\lambda L(n)}
= \E \left( e^{\lambda L(n)};\, \frac{L(n)}{\ell(n)}-1 \le \epsilon \right)
+ \E \left( e^{\lambda L(n)};\, \frac{L(n)}{\ell(n)}-1 > \epsilon \right)
=: \A_+(n) + \B_+(n).
\end{equation}
The first term is estimated as
\begin{equation}
\label{A+}
\A_+(n) \le e^{\lambda (1+\epsilon)\ell(n)} \,
\P\left(\frac{L(n)}{\ell(n)}-1 \le \epsilon \right),
\end{equation}
and so
\[
\frac{\ln \A_+(n)}{\ell(n)} \le \lambda (1+\epsilon) + o(1),
\]
implying that
\[
\limsup_{n \to \infty} \frac{\ln \A_+(n)}{\ell(n)} \le \lambda.
\]

For the second term we write
\begin{align*}
\B_+(n) &:=
\sum_{k=1}^\infty
\E\left( e^{\lambda L(n)};\, 1+k\epsilon < \frac{L(n)}{\ell(n)}
\le 1+ (k+1)\epsilon \right)
\\
& \le \sum_{k=1}^\infty e^{\lambda(1+ (k+1)\epsilon) \ell(n)}
\, \P\left(\frac{L(n)}{\ell(n)} > 1+k\epsilon \right)
\end{align*}
Observe now, from Lemma \eqref{Lbounds}, that
\[
\P(L(n) \ge k) =1-\P(L(n) <k)
\le 1- (1-p^k)^{n-k+1} \le (n-k+1) p^k \le n p^k,
\]
for all $0 \le k \le n$, and, trivially, for all $k > n$ also.
This implies that
\[
\P(L(n) > t) \le n p^t, \quad t \ge 0,
\]
and so
\[
P\left(\frac{L(n)}{\ell(n)} > 1+k\epsilon \right)
\le n\, p^{(1+k\epsilon)\ell(n)}
= n\cdot n^{-(1+k\epsilon)} = n^{-k\epsilon}.
\]
Therefore,
\begin{align*}
\B_+(n) & \le e^{\lambda(1+ \epsilon) \ell(n)} \,
\sum_{k=1}^\infty  e^{\lambda k \epsilon \ell(n)}\, n^{-k\epsilon}
\\
& = e^{\lambda(1+ \epsilon) \ell(n)} \,
\sum_{k=1}^\infty  n^{-\left(1-\frac{\lambda}{\ln(1/p)}\right)\,k\epsilon}
\\
&= e^{\lambda(1+ \epsilon) \ell(n)} \,
\left(n^{\left(1-\frac{\lambda}{\ln(1/p)}\right)\,\epsilon} -1\right)^{-1},
\end{align*}
whence
\[
\limsup_{n \to \infty} \frac{\ln \B_+(n)}{\ell(n)} \le \lambda.
\]
Since
\[
\limsup_{n \to \infty} \frac{\ln \E e^{\lambda L(n)}}{\ell(n)}
= \max\left\{\limsup_{n \to \infty} \frac{\ln \A_+(n)}{\ell(n)},\,
\limsup_{n \to \infty} \frac{\ln \B_+(n)}{\ell(n)} \right\}
\]
the result follows.

Suppose next that $\lambda < 0$. For $0< \epsilon < 1$, write
\[
\E e^{\lambda L(n)}
= \E \left( e^{\lambda L(n)};\, \frac{L(n)}{\ell(n)}-1 > -\epsilon \right)
+ \E \left( e^{\lambda L(n)};\, \frac{L(n)}{\ell(n)}-1 \le -\epsilon \right)
=: \A_-(n) + \B_-(n)
\]
For the first term we have
\[
\A_-(n) \le e^{\lambda (1-\epsilon) \ell(n)} \, \P\left(\frac{L(n)}{\ell(n)}-1 > -\epsilon \right)
\]
implying that
\[
\limsup_{n \to \infty} \frac{\ln \A_-(n)}{\ell(n)} \le \lambda.
\]
As for the second term,
\begin{align*}
\B_-(n) &= \sum_{k=1}^{\lfloor 1/\epsilon \rfloor -1}
 \E \left( e^{\lambda L(n)};\,
1-(k+1)\epsilon \le \frac{L(n)}{\ell(n)} < 1-k\epsilon \right)
\\
&\le
\sum_{k=1}^{\lfloor 1/\epsilon \rfloor -1}
e^{\lambda (1-(k+1)\epsilon)\ell(n)}\,
\P\left(\frac{L(n)}{\ell(n)} < 1-k\epsilon \right)
\end{align*}
Since there are only finitely many terms in the sum, we can simply
write
\begin{align*}
\limsup_{n \to \infty} \frac{\ln \B_-(n)}{\ell(n)}
& \le \max_{1\le k \le \lfloor 1/\epsilon\rfloor -1}
\left\{
\lambda(1-(k+1)\epsilon)
+ \limsup_{n \to \infty} \frac{1}{\ell(n)}
\ln \P\left(\frac{L(n)}{\ell(n)} < 1-k\epsilon \right)
\right\}
\\
&\le \max_{1\le k \le \lfloor 1/\epsilon\rfloor -1}
\left\{\lambda(1-(k+1)\epsilon) -\infty\right\} = -\infty,
\end{align*}
where $-\infty$ appears because of Lemma \ref{lemma:borrowed-from-wu} below.
We again conclude that $\limsup \ell(n)^{-1} \ln \E e^{\lambda L(n)}
\le \lambda$.
\end{proof}

The critical case is treated next.
\begin{lemma}\label{lemma:june-1}
When $\lambda=\ln(1/p),$ it holds that
\[
\lim_{n\rightarrow\infty}\frac{1}{\log_{1/p}n}\,
\ln,\,\mathbb{E}\exp\left\{\lambda\,
L(n)\right\}= 2\lambda.
\]
\end{lemma}
\begin{proof}
Fix sufficiently small $\epsilon > 0$.
Using the probability estimates of Lemma \ref{Lbounds}
we obtain that there exist positive constants $c_1$, $c_2$ such that
\begin{equation*}
c_2\, n^{-(1+k\epsilon)}\left(n+1-(1+k\epsilon)\ell(n)\right)
\le \P\left(\frac{L(n)}{\ell(n)}> 1+k\epsilon\right)
\leq
c_1\, n^{-(1+k\epsilon)}\left(n+1-(1+k\epsilon)\ell(n)\right)
\end{equation*}
uniformly over all $k$ such that
\begin{equation}
\label{Nn}
1\le k \le
\left\lfloor\frac{1}{\epsilon}\left(\frac{n}{\ell(n)}-1\right)\right\rfloor
=:N_n.
\end{equation}
We first obtain a lower bound.  From the estimate above,
\begin{align*}
\E e^{\lambda L(n)}
&\ge \E \left( e^{\lambda L(n)};\, \frac{L(n)}{\ell(n)} > 1+\epsilon \right)
\\
&\ge \sum_{k=1}^{N_n}
\E\left( e^{\lambda L(n)};\, 1+k\epsilon < \frac{L(n)}{\ell(n)}
\le 1+(k+1)\epsilon \right)
\\
&\ge \sum_{k=1}^{N_n}
\E\left( e^{\lambda \ell(n) (1+k\epsilon)};
\, 1+k\epsilon < \frac{L(n)}{\ell(n)} \le 1+(k+1)\epsilon \right).
\end{align*}
Since $\lambda = \ln(1/p)$ and $\ell(n) = (\ln n)/\ln(1/p)$ have
$ \exp(\lambda \ell(n)) = \exp(\ln n) =n$.
Hence
\begin{align*}
\E e^{\lambda L(n)}
&\ge n\, \sum_{k=1}^{N_n} n^{k\epsilon}
\,
\left[\P\left( \frac{L(n)}{\ell(n)} > 1+k\epsilon \right)
- \P\left( \frac{L(n)}{\ell(n)} > 1+(k+1)\epsilon \right)\right]
\\
& \ge n\, \sum_{k=1}^{N_n} n^{k\epsilon}
\,
\left[c_2\, n^{-(1+k\epsilon)}\left(n+1-(1+k\epsilon)\ell(n)\right)
-c_1\, n^{-(1+(k+1)\epsilon)}\left(n+1-(1+(k+1)\epsilon)\ell(n)\right)
\right]
\\
&= n\, \sum_{k=1}^{N_n}
\left[c_2\, n^{-1}\left(n+1-(1+k\epsilon)\ell(n)\right)
-c_1\, n^{-(1+\epsilon)}\left(n+1-(1+(k+1)\epsilon)\ell(n)\right)
\right]
\\
&=: n\, \S(n).
\end{align*}
Hence
\[
\frac{\ln \E e^{\lambda L(n)}}{\ell(n)}
\ge \frac{\ln n}{\ell(n)} + \frac{\ln \S(n)}{\ell(n)}
= \ln(1/p) + \ln(1/p)\, \frac{\ln \S(n)}{\ln n.}
\]
We now claim that the last ratio converges to $1$.
This follows by direct computation:
\[
\ln \S(n)
\sim \ln\left[ \frac{c_2 n}{2\epsilon \log_{1/p}n}
-\frac{c_1 n^{1-\epsilon}}{2\epsilon \log_{1/p}n} \right]
\sim \ln\left[ \frac{c_2 n}{2\epsilon \log_{1/p}n}\right]
= \ln n + o(\ln n).
\]
Hence we have proved a lower bound:
\[
\liminf_{n \to \infty} \frac{\ln \E e^{\lambda L(n)}}{\ell(n)}
\ge 2 \ln (1/p).
\]
To get an upper bound, we use the decomposition \eqref{dec}
as in the proof of Lemma \ref{lemma:less}, but with $\lambda=\ln (1/p)$.
The first term is estimated in precisely the same manner; see
\eqref{A+}. Hence
\begin{equation}
\label{AA}
\limsup_{n \to \infty} \frac{\ln \A_+(n)}{\ell(n)} \le \lambda = \ln(1/p).
\end{equation}
For the second term, we write
\begin{align*}
\B_+(n) &=
\E \left( e^{\lambda L(n)};\, \frac{L(n)}{\ell(n)}-1 > \epsilon \right)
\\
&=
\sum_{k=1}^{N_n}
\E\left( e^{\lambda L(n)};\, 1+k\epsilon < \frac{L(n)}{\ell(n)}
\le 1+ (k+1)\epsilon \right),
\end{align*}
where $N_n$ is as in \eqref{Nn}, giving
\begin{align*}
\B_+(n) &\le \sum_{k=1}^{N_n}
e^{\lambda \ell(n) [1+(k+1)\epsilon]}
\, \P\left(\frac{L(n)}{\ell(n)} > 1+k\epsilon\right)
\\
&= n^{1+\epsilon}\, \sum_{k=1}^{N_n} n^{k\epsilon}\,
\P\left(\frac{L(n)}{\ell(n)} > 1+k\epsilon\right)
\\
&\le n^{1+\epsilon}\, \sum_{k=1}^{N_n}
c_1\, n^{-1}\, (n+1-(1+k\epsilon)\ell(n))
\end{align*}
from which
\begin{align*}
\frac{\ln \B_+(n)}{\ell(n)} &\le
(1+\epsilon) \ln (1/p) + \ln (1/p)\, \frac{1}{\ln n} \, \ln \sum_{k=1}^{N_n}
c_1\, n^{-1}\, (n+1-(1+k\epsilon)\ell(n)).
\end{align*}
By direct computation,
\[
\ln \sum_{k=1}^{N_n}
c_1\, n^{-1}\, (n+1-(1+k\epsilon)\ell(n))
\sim \ln \left[\frac{c_1 n}{\epsilon \log_{1/p} n} \right]
= \ln n + o(\ln n).
\]
Combining the last two displays and letting $\epsilon \downarrow 0$
we obtain
\begin{equation}
\label{BB}
\limsup_{n \to \infty} \frac{\ln \B_+(n)}{\ell(n)} \le 2 \ln(1/p).
\end{equation}
From the decomposition \eqref{dec}, with the estimates \eqref{AA}
and \eqref{BB}, we conclude that
\[
\limsup_{n \to \infty} \frac{\ln \E e^{\lambda L(n)}}{\ell(n)}
= \max\left\{\limsup_{n \to \infty} \frac{\ln \A_+(n)}{\ell(n)},\,
\limsup_{n \to \infty} \frac{\ln \B_+(n)}{\ell(n)} \right\}
\le \max\{\ln(1/p), \, 2 \ln (1/p)\} = 2 \ln (1/p).
\]
\end{proof}

In order to study the asymptotic behavior of
$\mathbb{E}\exp\left\{\lambda L(n)\right\}$ when $\lambda>\ln(1/p)$, we
use the following result.
\begin{lemma}\label{lemma:a-LD-probability}
For fixed $0\le x\le 1$, it holds that
\[
\lim_{n\rightarrow\infty} \frac{1}{n}\,
\ln\,\P\left(\frac{L(n)}{n}\ge x\right)=-x\ln(1/p).
\]
\end{lemma}
\begin{proof}
We apply the inequalities of Lemma \ref{Lbounds}
with $k=\lceil nx \rceil$ and obtain
\begin{align*}
1-\left(1-qp^{\lceil nx \rceil}\right)^{n-\lceil nx \rceil+1}
\leq \P\left(\frac{L(n)}{n}\geq x\right)
\leq 1-\left(1-p^{\lceil nx \rceil}\right)^{n-\lceil nx \rceil+1}.
\end{align*}
Since, for $1-(1-a)^N \le Na$ for all $0 \le a \le 1$,
and since $1-(1-a)^N \ge (N-1) a$ for all sufficiently small $a\ge 0$,
we have that
\[
(n-\lceil nx \rceil)\, q p^{\lceil nx \rceil}
\le  \P\left(\frac{L(n)}{n}\geq x\right)
\le (n-\lceil nx \rceil+1)\, p^{\lceil nx \rceil},
\]
for all sufficiently large $n$.
Taking logarithms, dividing by $n$, and sending $n$ to $\infty$
finishes the proof.
\end{proof}

\begin{lemma}\label{lemma:infinity}
It holds that
\[
\lim_{n\rightarrow\infty}\frac{1}{n}\,\ln\,\mathbb{E}\exp\left\{\lambda\,
L(n)\right\}=\lambda-\ln(1/p),
\]
for $\lambda > \ln(1/p)$.
\end{lemma}
\begin{proof}
For the lower bound, fix  $0<x<1$, write
\begin{align*}
\E e^{\lambda L(n)}
\ge \E\left( e^{\lambda L(n)};\, \frac{L(n)}{n} > x \right)
\ge e^{\lambda x n} \, \P\left(\frac{L(n)}{n} > x \right),
\end{align*}
and use Lemma \ref{lemma:a-LD-probability}:
\[
\liminf_{n \to \infty} \frac{1}{n}\, \ln\, \E e^{\lambda L(n)}
\ge \lambda x - x \ln(1/p) \to \lambda - \ln(1/p),
\quad\text{as } x \to 1.
\]
For the upper bound, pick $\epsilon > 0$ and write
\begin{align*}
\E e^{\lambda L(n)}
&= \E\left( e^{\lambda L(n)};\, \frac{L(n)}{n} \le \epsilon \right)
+ \E\left( e^{\lambda L(n)};\, \frac{L(n)}{n} > \epsilon \right)
\\
&\le e^{\lambda \epsilon n}
+ \sum_{k=1}^{[1/\epsilon]-1}
\E\left( e^{\lambda L(n)};\, k\epsilon < \frac{L(n)}{n}\le (k+1)\epsilon\right)
\\
&\le e^{\lambda \epsilon n}
+ \sum_{k=1}^{[1/\epsilon]-1}
e^{\lambda (k+1)\epsilon n}\, \P\left(\frac{L(n)}{n} > k\epsilon\right)
\end{align*}
Hence (with $a \vee b := \max(a,b)$)
\begin{align*}
\limsup_{n \to \infty} \frac{1}{n}\, \E e^{\lambda L(n)}
&\le
(\lambda \epsilon) \vee
\max_{1\le k \le [1/\epsilon]-1}\left\{
\lambda (k+1)\epsilon - k \epsilon \ln(1/p)
\right\}
\\
& \le \lambda \epsilon + \lambda-\ln(1/p)
\to \lambda -\ln(1/p), \quad\text{ as } \epsilon \to 0.
\end{align*}
where we used Lemma \ref{lemma:a-LD-probability} again
and the assumption that $\lambda-\ln(1/p)>0$.
\end{proof}

\begin{lemma}[Theorem 1.1 in
\cite{Mao-Wang-Wu-2015}]\label{lemma:borrowed-from-wu}
For each $x>0,$ we have
\[
\lim_{n\rightarrow\infty}\frac{1}{\log_{1/p}n}\ln\,\,\mathbb{P}\left(\frac{L(n)}{\log_{1/p}n}\geq1+x\right)=-\ \ln(1/p).
\]
For every $0<x<1,$ we have
\[
\lim_{n\rightarrow\infty}\frac{1}{\log_{1/p}n}\ln\,\,\left[-\ln\,\,\mathbb{P}\left(\frac{L(n)}{\log_{1/p}n}\leq1-x\right)\right]=x \ln(1/p).
\]
\end{lemma}
\noindent Note that this lemma can be simply derived based on
Lemma \ref{Lbounds}, but what has actually been proved in
\cite{Mao-Wang-Wu-2015} is precise asymptotics without the logarithm.

\section{Large deviations principles}
\label{LDPsec}
We study the large deviations principles
announced in Corollaries \ref{th:longest-head-run} and
\ref{th-second:longest-head-run}.
Consider the logarithmic moment generating function of $L(n)/\log_{1/p}n$,
defined by
\[
\Lambda_n(\lambda)=\ln\,\,\mathbb{E}\exp\left\{\lambda\,
L(n)/\log_{1/p}n\right\},\quad \lambda\in\mathbb{R}.
\]
The proof of Corollary \ref{th:longest-head-run} is based on the
\textit{cumulant}, namely,
\begin{align}
\label{cumulant-definition}
\Lambda(\lambda):=\lim_{n\rightarrow \infty}\,
\frac{1}{\log_{1/p}n}\, \Lambda_n(\lambda \log_{1/p}n).
\end{align}
That this limit exists is a direct consequence of Theorem
\ref{proposition:Laplace-transform}:
\begin{proposition}
\label{proposition:cumulant}
The limit in \eqref{cumulant-definition} exists and is given by
\begin{equation*}
\Lambda(\lambda)=
\begin{cases}
	+\infty,& \lambda>\ln(1/p),\\
	2\lambda, & \lambda=\ln(1/p),\\
	\lambda,& \lambda<\ln(1/p).
\end{cases}
\end{equation*}
\end{proposition}
The Fenchel-Legendre transform of $\Lambda$ is the function
$x \mapsto \sup_{\lambda \in \R} [\lambda x - \Lambda(\lambda)]$
which (as an easy calculation shows)
is given by the function $\Lambda^*$ defined in \eqref{rate-function}:
\[
\sup_{\lambda \in \R} [\lambda x - \Lambda(\lambda)]
=
\Lambda^*(x)=
\begin{cases}
	+\infty,& x<1,\\
	(x-1)\ln(1/p),& x \ge 1.
\end{cases}
\]

\begin{proof}[Proof of Corollary \ref{th:longest-head-run}]
To prove the upper bound \eqref{LDP-upper-bound} we apply
the G\"{a}rtner-Ellis theorem
(cf. Section 2.3 in \cite{Dembo-Zeitouni-2009}).
For the lower bound \eqref{LDP-lower-bound}, we must give
a separate argument.
It suffices to prove that for a fixed point $y>1,$
\begin{align}
\label{end-near}
\lim_{\delta\rightarrow0}\liminf_{n\rightarrow\infty}
\frac{1}{\log_{1/p}n}\, \ln\,
\P\left(\frac{L(n)}{\log_{1/p}n}\in B_{y,\delta}\right)
\geq-(y-1)\ln(1/p),
\end{align}
where $B_{y,\delta}$ is the open ball centered at $y$ with a radius $\delta.$ To
achieve \eqref{end-near}, we write
\begin{align*}
\mathbb{P}\left(\frac{L(n)}{\log_{1/p}n}\in
B_{y,\delta}\right)=\mathbb{P}\left(\frac{L(n)}{\log_{1/p}n}>y-\delta\right)-\mathbb{P}\left(\frac{L(n)}{\log_{1/p}n}\geq
y+\delta\right).
\end{align*}
In order to analyze the logarithm, we apply an inequality in the form
$\ln(a-b)\geq \ln(a)-\frac{b}{a-b}$ for $a>b>0.$ Therefore,
\begin{multline}
\lim_{\delta\to 0}\liminf_{n\to\infty} \frac{1}{\ell(n)}\,
\ln\,\P\left(\frac{L(n)}{\ell(n)}\in B_{y,\delta}\right)
\\
\ge \lim_{\delta\to 0}\liminf_{n\to \infty} \frac{1}{\ell(n)}
\left(\ln\,\P\left(\frac{L(n)}{\ell(n)}>y-\delta\right)
-\frac{\P\left(\frac{L(n)}{\ell(n)}\ge y+\delta\right)}
{\P\left(\frac{L(n)}{\ell(n)}> y-\delta\right)
-\P\left(\frac{L(n)}{\ell(n)}\ge y+\delta\right)}\right).
\label{end-1}
\end{multline}
We can apply Lemma \ref{lemma:borrowed-from-wu} to handle the first limit as
follows
\begin{equation}
\label{end-2}
\lim_{\delta\to 0} \liminf_{n\to \infty} \frac{1}{\ell(n)}\,
\ln\,\P\left(\frac{L(n)}{\ell(n)} > y-\delta\right)
=\lim_{\delta\to 0} -(y-1-\delta)\ln(1/p)
= -(y-1)\ln(1/p).
\end{equation}
For the last ratio term in \eqref{end-1}, it follows from applying Lemma
\ref{lemma:borrowed-from-wu} twice that
\[
\P\left(\frac{L(n)}{\ell(n)}\ge  y+\delta\right)\leq
\exp\left\{\left[-(y-1+\delta)\ln(1/p)+\varepsilon_1\right]\ell(n)\right\}
\]
and
\[
\P\left(\frac{L(n)}{\ell(n)}> y-\delta\right)\ge
\exp\left\{\left[-(y-1-\delta)\ln(1/p)-\varepsilon_2\right]\ell(n)\right\}
\]
for sufficiently small $\varepsilon_1>0$ and $\varepsilon_2>0.$ Thus, assuming
$2\delta\ln(1/p)-\varepsilon_1-\varepsilon_2>0,$
\begin{multline}
\frac{\P\left(\frac{L(n)}{\ell(n)}\geq
y+\delta\right)}{\P\left(\frac{L(n)}{\ell(n)}>
y-\delta\right)-\P\left(\frac{L(n)}{\ell(n)}\geq y+\delta\right)}
=\frac{1}{\P\left(\frac{L(n)}{\ell(n)}>
y-\delta\right)/\P\left(\frac{L(n)}{\ell(n)}\geq
y+\delta\right)-1}
\\
\le \frac{1}{e^{(2\delta\ln(1/p)-\varepsilon_1-\varepsilon_2)\ell(n)}-1}
\to 0 ,\quad \text{as } n \to \infty.
\label{end-3}
\end{multline}
Then \eqref{end-near} follows by substituting
\eqref{end-2} and \eqref{end-3} back into \eqref{end-1}.
\end{proof}

We now pass to the second large deviations principle.
Consider the logarithmic moment generating function of $L(n)/n$:
\[
\widetilde{\Lambda}_n(\lambda)
=\ln\,\E\exp\left\{\lambda\, L(n)/n\right\},\quad \lambda\in\mathbb{R},
\]
and define its cumulant by
\begin{equation}
\label{cumdef2}
\widetilde{\Lambda}(\lambda)
:=\lim_{n\rightarrow \infty}\,\frac{1}{n}\, \Lambda_n(\lambda n).
\end{equation}
It is again Theorem \ref{proposition:Laplace-transform}
that is responsible for the existence of the cumulant:
\begin{proposition}
The limit in \eqref{cumdef2} exists and is given by the formula
\[
\widetilde{\Lambda}(\lambda)=
\begin{cases}
	\lambda-\ln(1/p),& \lambda\geq\ln(1/p),\\
	0,& \lambda<\ln(1/p).
\end{cases}
\]
\end{proposition}
An easy calculation shows that
\[
\sup_{\lambda \in \R} [\lambda x - \widetilde \Lambda(\lambda)]
=
\widetilde \Lambda^*(x)=
\begin{cases}
	+\infty,& x<0,\\
	x\ln(1/p),& 0\leq x\leq1,\\
	+\infty,& x>1.
\end{cases}
\]
which is the function announced in \eqref{rate-function-second}.
The proof of Corollary \ref{th-second:longest-head-run}
now proceeds along the same lines as that
of Corollary \ref{th:longest-head-run} and is therefore omitted.

\section{Exponential functionals}
\label{subsec:proof-cor-third}
The proof of Corollary \ref{th-third:general-Laplace}
is straightforward and follows from
Varadhan's integral lemma (cf.\ \cite[Section 4.3]{Dembo-Zeitouni-2009}).

We next verify that the function $f(x)=c\cdot x^\alpha$,
$0<\alpha<1,$ satisfies the condition (\ref{A1}) in Corollary
\ref{th-third:general-Laplace}. Without loss of generality, we assume $c>0$ and
obtain
\begin{multline*}
\frac{1}{\ell(n)}\ln\E\left[\exp\left(\ell(n)
f(\frac{L(n)}{\ell(n)})\right);\,
f(\frac{L(n)}{\ell(n)})\geq m\right]
\\
\begin{aligned}
&=\frac{1}{\ell(n)}\ln\E\left[\exp\left(c \ell(n)
(\frac{L(n)}{\ell(n)})^\alpha\right);\,
(\frac{L(n)}{\ell(n)})^\alpha> m \right]
\\
&= \frac{1}{\ell(n)}\ln \sum_{k=0}^\infty
\E\left[\exp\left(c \ell(n)
(\frac{L(n)}{\ell(n)})^\alpha\right);\,
m+k<(\frac{L(n)}{\ell(n)})^\alpha\leq m+(k+1) \right]
\\
&\leq\frac{1}{\ell(n)}\ln \sum_{k=0}^\infty
e^{c \ell(n) (m+k+1)} \,\P\left(m+k<(\frac{L(n)}{\ell(n)})^\alpha\right)
\\
&=c(m+1)+\frac{1}{\ell(n)}\ln \sum_{k=0}^\infty
e^{ck\ell(n)} \P\left((m+k)^{1/\alpha}<\frac{L(n)}{\ell(n)}\right).
\end{aligned}
\end{multline*}
We now apply Lemma \ref{Lbounds} with $k=\lceil(m+k)^{1/\alpha} \ell(n)\rceil+1$
and obtain
\begin{align*}
\P\left(\frac{L(n)}{\ell(n)} > (m+k)^{1/\alpha} \right)
&= \P\left(\L(n) > \lceil \ell(n) (m+k)^{1/\alpha} \rceil \right)
\\
&= 1-\P\left(\L(n) < \lceil \ell(n) (m+k)^{1/\alpha} \rceil+1 \right)
\\
&\le 1-(1-p^{\lceil \ell(n) (m+k)^{1/\alpha} \rceil+1})^{n-\lceil \ell(n) (m+k)^{1/\alpha} \rceil}
\\
&\le (n-\lceil \ell(n) (m+k)^{1/\alpha} \rceil)
p^{\lceil \ell(n) (m+k)^{1/\alpha} \rceil+1}
\\
&\le n p^{\ell(n) (m+k)^{1/\alpha}}
\\
&= n^{1-(m+k)^{1/\alpha}}.
\end{align*}
Combining previous two estimates gives
\begin{align*}
\frac{1}{\ell(n)}& \ln  \mathbb{E}\left[\exp\left(\ell(n)
f(\frac{L(n)}{\ell(n)})\right);\,
f(\frac{L(n)}{\ell(n)})\geq m \right]
\\
&\leq c(m+1)+\frac{1}{\ell(n)}\ln  \left(\sum_{k=0}^\infty
e^{ck\ell(n)}  n^{-(m+k)^{1/\alpha}+1}\right)
\\
&=c(m+1)+\frac{1}{\ell(n)}\ln  \left(\sum_{k=0}^\infty
n^{ck/\ln(1/p)}  n^{-(m+k)^{1/\alpha}+1}\right)
\\
&\leq c(m+1)+\frac{1}{\ell(n)}\ln  \left(\sum_{k=0}^\infty
n^{ck/\ln(1/p)}  n^{-(m^{1/\alpha}+k^{1/\alpha})/2+1}\right)
\\
&=c(m+1)-\frac{(m^{1/\alpha}-1)
\ln(1/p)}{2}+\frac{1}{\ell(n)}\ln  \left(\sum_{k=0}^\infty
n^{ck/\ln(1/p)}  n^{-k^{1/\alpha}/2}\right)
\\
&\rightarrow c(m+1)-\frac{(m^{1/\alpha}-1)  \ln(1/p)}{2},\quad \text{ as
}n\rightarrow\infty \quad(\text{since }\alpha<1).
\end{align*}
Therefore (\ref{A1}) follows by taking $m\rightarrow\infty.$

With $f(x) = tx^\alpha$, $t > 0$, $0<\alpha<1$,
we have
\[
\max_{x \in \R} \{f(x)-\Lambda^*(x)\}
= \max_{x \ge 1} \{tx^\alpha-\lambda_p (x-1)\},
\]
where $\lambda_p:=\ln(1/p)$, for brevity.
There are two cases:
\\
{\em Case 1:} $t > \ln(1/p)/\alpha$. Then the maximum above
is achieved at $x^* = (\alpha t/\lambda_p)^{1/(1-\alpha)}$
and equals
\[
t^{\frac{1}{1-\alpha}}\,
\lambda_p^{\frac{-\alpha}{1-\alpha}}\,C_\alpha+\lambda_p,
\]
where $C_\alpha$ is the positive quantity
\[
C_\alpha = \alpha^{\frac{\alpha}{1-\alpha}} - \alpha^{\frac{1}{1-\alpha}}.
\]
Since $\ell(n) f(L(n)/\ell(n)) = t \ell(n)^{1-\alpha} L(n)^\alpha
= t \lambda_p^{\alpha-1} (\ln n)^{1-\alpha} L(n)^\alpha$,
Corollary \ref{th-third:general-Laplace}
gives
\[
\ln \E\left[ e^{t \lambda_p^{\alpha-1} (\ln n)^{1-\alpha} L(n)^\alpha}\right]
\sim \frac{\ln n}{\lambda_p}
\left(t^{\frac{1}{1-\alpha}}\,
\lambda_p^{\frac{-\alpha}{1-\alpha}}\,C_\alpha+\lambda_p\right)
= (\ln n) \left(t^{\frac{1}{1-\alpha}}\,
\lambda_p^{\frac{-1}{1-\alpha}}\,C_\alpha+1\right).
\]
\\
{\em Case 2:} $t \le \ln(1/p)/\alpha$.
Then the maximum is achieved at $x^*=1$ and equals $t$.
Hence
\[
\ln \E\left[ e^{t \lambda_p^{\alpha-1} (\ln n)^{1-\alpha} L(n)^\alpha}\right]
\sim \frac{t}{\lambda_p} \ln n.
\]
The expressions become neater upon a change of variables
and are summarized thus
\begin{corollary}
For all $t >0$, for all $0<\alpha<1$, as $n \to \infty$,
\[
\ln \E\left[ e^{t\, (\ln n)^{1-\alpha} L(n)^\alpha}\right]
\sim
\begin{cases}
\displaystyle
\frac{t}{\ln^\alpha(1/p)} \ln n,
& \text{ if } t\le \frac{\ln^\alpha(1/p)}{\alpha}
\\[4mm]
\displaystyle
\left[ \left(\frac{t}{\ln^\alpha(1/p)}\right)^{\frac{1}{1-\alpha}}
\left(\alpha^{\frac{\alpha}{1-\alpha}} - \alpha^{\frac{1}{1-\alpha}}\right)+1
\right] \ln n ,
& \text{ otherwise.}
\end{cases}
\]
\end{corollary}

\section{An applications to inference}
\label{APPLsec}
Let us consider a classical problem in confidence intervals.
Let $\{X_k\}_{1\leq k\leq n}$ be an i.i.d.\ random sample from a
Bernoulli population $X$ with
$\mathbb{P}(X=1)=p$ and $\mathbb{P}(X=0)=1-p, \,\,0<p<1.$ Our aim in this
section is to construct a $100(1-\alpha)\%$ confidence interval for $p$ with a
given significance level $\alpha,$ when $p$ is close to $1$ (or $0$) and $n$ is
not very large.

The normal approximation to the binomial random variable
$K:=\sum_{i=1}^nX_i$ does not work well when $p$ is close to $1$ (or $0$).
Nevertheless, there are several alternatives in this case:
Wilson's score interval \cite{Wilson-1927}, the Clopper-Pearson interval
\cite{Clopper-Pearson-1934}, and others (such as Jeffreys interval,
Agresti-Coull Interval etc.). In this section, we propose another confidence
interval based on the longest head run $L(n)$ with the help of Corollary
\ref{th:longest-head-run}. It turns out that this type of confidence intervals
works much better than others.

To construct such confidence intervals, on one hand it comes from
Corollary \ref{th:longest-head-run} that, for each $x>0,$
$$\lim_{n\rightarrow\infty}\frac{1}{\log_{1/p}n}\ln\,\,\mathbb{P}\left(\frac{L(n)}{\log_{1/p}n}\geq1+x\right)=-x\cdot\ln(1/p).$$
On the other hand, Lemma \ref{lemma:borrowed-from-wu} below states that, for
every $0<x<1,$
$$\lim_{n\rightarrow\infty}\frac{1}{\log_{1/p}n}\ln\,\,\left[-\ln\,\,\mathbb{P}\left(\frac{L(n)}{\log_{1/p}n}\leq1-x\right)\right]=x\cdot\ln(1/p).$$
Combining these two asymptotics gives a $100(1-\alpha)\%$ confidence interval of
$p$ as follows:
\begin{align}
\label{confidence-interval}
I_p=\left(\exp\left\{-\frac{\ln(n)-\ln(\alpha/2)}{\widehat{L}(n)}\right\},\quad\exp\left\{-\frac{\ln(n)-\ln(-\ln(\alpha/2))}{\widehat{L}(n)}\right\}\right)
\end{align}
where $\widehat{L}(n)$ is a point estimate of $L(n).$
A reasonable point estimate of $L(n)$ is
\[
\widehat{L}(n)=L_{\text{obs}}(n)-\left[\log_{1/\widehat{p}}(1-\widehat{p})+\log_{1/\widehat{p}}(e^\gamma)-\frac{1}{2}\right]
\]
with $L_{\text{obs}}(n)$ being the observed longest head run in $n$ trials, and
$\widehat{p}:=k/n$ being the sample proportion. To see this, firstly we know that in
the long run $L(n)/\log_{1/p}n\rightarrow1,$ therefore we want that an estimate
satisfies $\mathbb{E}\widehat{L}(n)\rightarrow \log_{1/p}n.$ Secondly, it follows
from the mean \eqref{june-10-mean} that
$\mathbb{E}\widehat{L}(n)=\log_{1/p}n+\left[\log_{1/p}(1-p)+\log_{1/p}(e^\gamma)\right]-\left[\log_{1/\widehat{p}}(1-\widehat{p})+\log_{1/\widehat{p}}(e^\gamma)\right]+\varepsilon(n),$
which is quite close to $\log_{1/p}n.$ This explains that
\eqref{confidence-interval} is an appropriate confidence interval for $p.$

Below we have simulations for the derived confidence interval $I_p$ in
\eqref{confidence-interval} when $p$ is close to $1$ (the case $p$ is close to
$0$ can be similarly handled), and we make several comparisons with Wilson score
intervals and Clopper-Pearson intervals. Based on the simulations
(see Table 1),
it is evident that our confidence interval
\eqref{confidence-interval} works much better than others when $p$ is close to
$1$ and $n$ is not very large.
In Table 2,
for larger $p$ and $n$ we
apply the normal approximation to the Binomial random variable. In this case it
turned out that lower bound of the normal approximation intervals works better
than Wilson score intervals and Clopper-Pearson intervals, but the upper bound
does not. In any case, our confidence interval \eqref{confidence-interval} still
works the best among them.

\begin{center}
\small
{\bf Table 1.} Wilson score interval (WS) --- Clopper-Pearson interval (CP)
--- Longest run interval (LR)     
\\[2mm]
\begin{tabular}{r | c c c c c}
\hline
& & $p=0.9500$ & $n=200$ & $\alpha=0.05$ & \\
\hline
& $\widehat{p}=0.9650$ & $\widehat{p}=0.9450$ & $\widehat{p}=0.9600$ & $\widehat{p}=0.9500$ &
$\widehat{p}=0.9700$\\
WS: & $(0.9295, 0.9829)$ & $(0.9042, 0.9690)$ & $(0.9231, 0.9796)$ & $(0.9104,
0.9726)$ & $(0.9361, 0.9862)$\\
CP: & $(0.9292, 0.9858)$ & $(0.9037, 0.9722)$ & $(0.9227, 0.9826)$ & $(0.9100,
0.9758)$ & $(0.9358, 0.9889)$\\
LR: & $(0.9329, 0.9696)$ & $(0.9145, 0.9611)$ & $(0.9243, 0.9656)$ & $(0.9325,
0.9694)$ & $(0.9484, 0.9767)$\\
\hline
\end{tabular}
\begin{tabular}{r | c c c c c}
\hline
& & $p=0.98$ & $n=200$ & $\alpha=0.05$ & \\
\hline
& $\widehat{p}=0.9800$ & $\widehat{p}=0.9850$ & $\widehat{p}=0.9700$ & $\widehat{p}=0.9800$ &
$\widehat{p}=0.9750$\\
WS: & $(0.9497, 0.9922)$ & $(0.9568, 0.9949)$ & $(0.9361, 0.9862)$ & $(0.9497,
0.9922)$ & $(0.9428, 0.9893)$\\
CP: & $(0.9496, 0.9945)$ & $(0.9568, 0.9969)$ & $(0.9358, 0.9889)$ & $(0.9496,
0.9945)$ & $(0.9426, 0.9918)$\\
LR: & $(0.9657, 0.9846)$ & $(0.9751, 0.9889)$ & $(0.9578, 0.9810)$ & $(0.9703,
0.9867)$ & $(0.9606, 0.9821)$\\
\hline
\end{tabular}
\end{center}

\begin{center}
\small
{\bf Table 2.} Wilson score interval (WS) --- Clopper-Pearson interval (CP)
--- Longest run interval (LR)\\
--- Normal approximation (N)\\[2mm]    
\begin{tabular}{r | c c c c c}
\hline
& & $p=0.9950$ & $n=1000$ & $\alpha=0.05$ & \\
\hline
& $\widehat{p}=0.9950$ & $\widehat{p}=0.9940$ & $\widehat{p}=0.9950$ & $\widehat{p}=0.9960$ &
$\widehat{p}=0.9960$\\
N: & $(0.9906, 0.9994)$ & $(0.9892, 0.9988)$ & $(0.9906, 0.9994)$ & $(0.9921,
0.9999)$ & $(0.9921, 0.9999)$\\
WS: & $(0.9883, 0.9979)$ & $(0.9870, 0.9972)$ & $(0.9883, 0.9979)$ & $(0.9898,
0.9984)$ & $(0.9898, 0.9984)$\\
CP: & $(0.9884, 0.9984)$ & $(0.9870, 0.9978)$ & $(0.9884, 0.9984)$ & $(0.9898,
0.9989)$ & $(0.9898, 0.9989)$\\
LR: & $(0.9915, 0.9955)$ & $(0.9909, 0.9952)$ & $(0.9919, 0.9957)$ & $(0.9941,
0.9969)$ & $(0.9938, 0.9967)$\\
\hline
\end{tabular}
\end{center}

\section{Open problems}
\label{OPENsec}
A problem for future research
would be the study of a large deviation principle for
the random-dimensional random vector $R(n)=
\left(R_1(n), R_2(n), \ldots, R_{L(n)}(n)\right)$
of counts of successive runs of all lengths. That is,
let $R_\ell(n)$ be the number of head runs of length $\ell$
up to the $n$-th coin toss. Distributional relations for $R(n)$
were studied in \cite{Holst-Konstantopoulos-2015}.

It would further be interesting to obtain large large deviation principles
for longest runs in a Markov chain. In other words, assume that $(X_n)$
is a Markov chain with finite (or countable) state space $S$ and let
$L(x,n)$ be the longest sojourn time at a state $x \in S$ before time $n$.
Although there are Stein-Chen type estimates
\cite{Novak-1992,Zhang-Wu-2013} for the distribution
of such quantities, the errors in these estimates are too big for the
study of a large deviation principle.
We would like to obtain an LDP for $L(x,n)$ or for the vector
$(L(x,n), \, x \in S)$ which would, by contraction principle,
give us an LDP for $L(n) := \sup_{x \in S} L(x,n)$.

\end{document}